\documentclass[12pt,reqno]{amsart}
\usepackage{fullpage,url,amssymb,enumerate,colonequals}
\usepackage[all]{xy} 
\usepackage{mathrsfs} 

\usepackage[OT2,T1]{fontenc}
\DeclareSymbolFont{cyrletters}{OT2}{wncyr}{m}{n}
\DeclareMathSymbol{\Sha}{\mathalpha}{cyrletters}{"58}

\usepackage{color}

\newcommand{\defi}[1]{\textsf{#1}} 

\newcommand{\Aff}{\mathbb{A}}

\newcommand{\F}{\mathbb{F}}

\newcommand{\PP}{\mathbb{P}}

\newcommand{\R}{\mathbb{R}}
\newcommand{\Xtilde}{\widetilde{X}}
\newcommand{\Z}{\mathbb{Z}}

\newcommand{\Vbar}{{\overline{V}}}
\newcommand{\Xbar}{{\overline{X}}}

\newcommand{\calC}{\mathcal{C}}

\newcommand{\calP}{\mathcal{P}}
\newcommand{\calQ}{\mathcal{Q}}

\newcommand{\calX}{\mathcal{X}}

\newcommand{\OO}{\mathscr{O}}
\newcommand{\TT}{\mathscr{T}}
\newcommand{\VV}{\mathscr{V}}
\newcommand{\XX}{\mathscr{X}}
\newcommand{\YY}{\mathscr{Y}}

\DeclareMathOperator{\Frob}{Frob}
\DeclareMathOperator{\Gal}{Gal}

\DeclareMathOperator{\homog}{homog}

\DeclareMathOperator{\horiz}{horiz}
\DeclareMathOperator{\Id}{Id}

\DeclareMathOperator{\Irr}{Irr}

\DeclareMathOperator{\NE}{NE}

\DeclareMathOperator{\PIC}{\bf Pic}

\DeclareMathOperator{\Proj}{Proj}

\DeclareMathOperator{\re}{Re}

\DeclareMathOperator{\Spec}{Spec}

\DeclareMathOperator{\Stab}{Stab}

\newcommand{\et}{{\operatorname{et}}}

\newcommand{\red}{{\operatorname{red}}}
\newcommand{\sing}{{\operatorname{sing}}}
\newcommand{\smooth}{{\operatorname{smooth}}}

\newcommand{\GL}{\operatorname{GL}}

\newcommand{\injects}{\hookrightarrow}
\newcommand{\intersect}{\cap} 
\newcommand{\Intersection}{\bigcap} 
\newcommand{\isom}{\simeq}

\newcommand{\tensor}{\otimes} 
\newcommand{\union}{\cup} 
\newcommand{\Union}{\bigcup} 

\newtheorem{theorem}{Theorem}[section]
\newtheorem{lemma}[theorem]{Lemma}
\newtheorem{corollary}[theorem]{Corollary}
\newtheorem{proposition}[theorem]{Proposition}

\theoremstyle{definition}

\newtheorem{example}[theorem]{Example}

\theoremstyle{remark}
\newtheorem{remark}[theorem]{Remark}
\newtheorem{remarks}[theorem]{Remarks}

\makeatletter
\newcommand{\interitemtext}[1]{%
\begin{list}{}
{\itemindent=0mm\labelsep=0mm
\labelwidth=0mm\leftmargin=0mm
\addtolength{\leftmargin}{-\@totalleftmargin}}
\item #1
\end{list}}
\makeatother

\usepackage{microtype}

\usepackage[
	pagebackref,
	pdfauthor={Bjorn Poonen}, 
]{hyperref}
\usepackage[alphabetic,lite]{amsrefs} 

\begin{document}

\title[Bertini irreducibility]{Bertini irreducibility theorems over finite fields}
\author{Fran\c{c}ois Charles}
\address{Department of Mathematics, Massachusetts Institute of Technology, Cambridge, MA 02139-4307, USA; and CNRS, Laboratoire de Math\'ematiques d'Orsay, Universit\'e Paris-Sud, 91405 Orsay CEDEX, France}
\email{francois.charles@math.u-psud.fr}
\author{Bjorn Poonen}
\thanks{This research was supported in part by National Science Foundation grants DMS-1069236 and DMS-1601946 and a grant from the Simons Foundation (\#402472 to Bjorn Poonen).  The article was originally published in \emph{J.~Amer.\ Math.\ Soc.}\ \textbf{29} (2016), no.~1, 81--94.}
\address{Department of Mathematics, Massachusetts Institute of Technology, Cambridge, MA 02139-4307, USA}
\email{poonen@math.mit.edu}
\urladdr{\url{http://math.mit.edu/~poonen/}}
\date{July 6, 2014; corrected May 3, 2017}

\begin{abstract}
Given a geometrically irreducible subscheme $X \subseteq \PP^n_{\F_q}$
of dimension at least $2$,
we prove that the fraction of degree~$d$ hypersurfaces $H$
such that $H \intersect X$ is geometrically irreducible
tends to $1$ as $d \to \infty$.
We also prove variants in which $X$ is over an
extension of $\F_q$, and in which the immersion 
$X \to \PP^n_{\F_q}$ is replaced by a more general morphism.
\end{abstract}

\maketitle

\section{Introduction}\label{S:introduction}

The classical Bertini theorems over an infinite field $k$
state that if a subscheme $X \subseteq \PP^n_k$
has a certain property 
(smooth, geometrically reduced, geometrically irreducible), 
then a sufficiently general hyperplane section over $k$ has
the property too.
In \cite{Poonen2004-bertini}, 
an analogue of the Bertini smoothness theorem for a finite field $\F_q$ 
was proved, in which hyperplanes 
were replaced by hypersurfaces of degree tending to infinity.

The goal of the present article is to prove Bertini irreducibility
theorems over finite fields.
The proof of the Bertini irreducibility theorem over infinite fields $k$
\cite{Jouanolou1983}*{Th\'eor\`eme~6.3(4)} 
relies on the fact that 
a dense open subscheme of $\PP^n_k$ has a $k$-point, 
so the proof fails over a finite field: see also the
end of the introduction in \cite{Benoist2011},
where this problem is mentioned.
The proof of the Bertini smoothness theorem over finite fields
in \cite{Poonen2004-bertini} 
depends crucially on the fact that smoothness can be checked 
analytically locally, one closed point at a time;
in contrast, irreducibility and geometric irreducibility are not local
properties.
Therefore our proof must use ideas beyond those used in proving
the earlier results.
Indeed, our proof requires ingredients that are perhaps unexpected: 
resolution of singularities for surfaces, cones of curves in a surface,
and the function field Chebotarev density theorem.

\subsection{Results for subschemes of projective space}

Let $\F_q$ be a finite field of size $q$.
Let $\F$ be an algebraic closure of $\F_q$.
Let $S=\F_q[x_0,\ldots,x_n]$ be the homogeneous coordinate ring 
of $\PP^n_{\F_q}$,
let $S_d \subset S$
be the $\F_q$-subspace of homogeneous polynomials of degree $d$,
and let $S_{\homog}=\Union_{d=0}^\infty S_d$.
For each $f \in S_{\homog}$, let $H_f$ be the subscheme
$\Proj(S/(f)) \subseteq \PP^n$,
so $H_f$ is a hypersurface (if $f$ is not constant).
Define the \defi{density} of a subset $\calP \subseteq S_{\homog}$
by
\[
	\mu(\calP):= \lim_{d \to \infty} 
		\frac{\#(\calP \cap S_d)}{\# S_d},
\]
if the limit exists.
Define \defi{upper} and \defi{lower density} similarly,
using $\limsup$ or $\liminf$ in place of $\lim$.

\begin{theorem}
\label{T:geometrically irreducible}
Let $X$ be a geometrically irreducible subscheme of $\PP^n_{\F_q}$.
If $\dim X \ge 2$, then the density of 
\[
	\{ f \in S_{\homog} : 
	\textup{$H_f \intersect X$ is geometrically irreducible}\}
\]
is $1$.
\end{theorem}

We can generalize by requiring only that $X$ be defined over $\F$;
we still intersect with hypersurfaces over $\F_q$, however.
Also, we can relax the condition of geometric irreducibility
in both the hypothesis and the conclusion: the result is
Theorem~\ref{T:bijection} below.
To formulate it, we introduce a few definitions.
Given a noetherian scheme $X$, let $\Irr X$ be its set of
irreducible components.
If $f \in S_{\homog}$ and $X$ is a subscheme of $\PP^n_L$ 
for some field extension $L \supseteq \F_q$,
let $X_f$ be the $L$-scheme $(H_f)_L \intersect X$.

\begin{theorem}
\label{T:bijection}
Let $X$ be a subscheme of $\PP^n_{\F}$ 
whose irreducible components are of dimension at least~$2$.
For $f$ in a set of density~$1$,
there is a bijection $\Irr X \to \Irr X_f$ sending $C$ to $C \intersect X_f$.
\end{theorem}

\begin{remarks}\hfill
\begin{enumerate}[\upshape (a)]
\item 
For some $f$, 
the specification $C \mapsto C \intersect X_f$
does not even define a map $\Irr X \to \Irr X_f$;
i.e., $C \in \Irr X$ does not imply $C \intersect X_f \in \Irr X_f$.
\item 
If the specification does define a map, then the map is surjective.
\item
Given a geometrically irreducible subscheme $X \subseteq \PP^n_{\F_q}$,
Theorem~\ref{T:geometrically irreducible} for $X$
is equivalent to Theorem~\ref{T:bijection} for $X_\F$.
Thus Theorem~\ref{T:bijection} is more general than
Theorem~\ref{T:geometrically irreducible}.
\item
It may seem strange to intersect a scheme $X$ over $\F$
with hypersurfaces defined over $\F_q$,
but one advantage of Theorem~\ref{T:bijection} is that it implies
the analogue of Theorem~\ref{T:geometrically irreducible}
in which ``geometrically irreducible'' is replaced by ``irreducible''
in both places.
More generally, Theorem~\ref{T:bijection} implies 
an $\F_q$-analogue of itself,
namely Corollary~\ref{C:irreducible component bijection} below.
\end{enumerate}
\end{remarks}

\begin{corollary}
\label{C:irreducible component bijection}
Let $X$ be a subscheme of $\PP^n_{\F_q}$ 
whose irreducible components are of dimension at least~$2$.
For $f$ in a set of density~$1$,
there is a bijection $\Irr X \to \Irr X_f$ sending $C$ to $C \intersect X_f$.
\end{corollary}

\begin{proof}
Apply Theorem~\ref{T:bijection} to $X_\F$,
and identify $\Irr X$ with the set of $\Gal(\F/\F_q)$-orbits
in $\Irr X_{\F}$, and likewise for $X_f$.
\end{proof}

\subsection{Results for morphisms to projective space}

In \cite{Jouanolou1983} 
one finds a generalization (for an infinite field $k$)
in which the subscheme $X \subseteq \PP^n$ is replaced by
a $k$-morphism $\phi \colon X \to \PP^n$.
Specifically, Th\'eor\`eme~6.3(4) of~\cite{Jouanolou1983} 
states that for a morphism of finite-type $k$-schemes
$\phi \colon X \to \PP^n$ 
with $X$ geometrically irreducible and $\dim \overline{\phi(X)} \ge 2$,
almost all hyperplanes $H \subseteq \PP^n$
are such that $\phi^{-1} H$ is geometrically irreducible;
here $\overline{\phi(X)}$ is the Zariski closure of $\phi(X)$ in $\PP^n$,
and ``almost all'' refers to a dense open subset
of the moduli space of hyperplanes.
The following example shows that we cannot expect 
such a generalization to hold for a density~$1$ set of hypersurfaces 
over a finite field.

\begin{example}
Let $n \ge 2$, and let $\phi \colon X \to \PP^n_{\F_q}$ 
be the blowing up at a point $P \in \PP^n(\F_q)$.
The density of the set of $f$ that vanish at $P$ is $1/q$,
and for any such nonzero $f$,
the scheme $\phi^{-1} H_f$ is 
the union of the exceptional divisor $\phi^{-1} P$
and the strict transform of $H_f$,
so it is not irreducible, 
and hence not geometrically irreducible.
\end{example}

We can salvage the result by disregarding
irreducible components of $X_{\F}$ that are contracted to a point.
To state the result, 
Theorem~\ref{T:inverse image bijection},
we introduce the following terminology:
given a morphism $\phi \colon X \to \PP^n$,
a subscheme $Y$ of $X$ (or $X_{\F}$) is \defi{vertical} 
if $\dim \overline{\phi(Y)}=0$, and \defi{horizontal} otherwise.
Let $\Irr_{\horiz} Y$ be the set of horizontal irreducible
components of $Y$, and let $Y_{\horiz}$ be their union.
Define $X_f \colonequals \phi^{-1} H_f$,
viewed as a scheme over the same extension of $\F_q$ as $X$;
this definition extends the earlier one.

\begin{theorem}
\label{T:inverse image bijection}
Let $X$ be a finite-type $\F$-scheme.
Let $\phi \colon X \to \PP^n_\F$ be an $\F$-morphism such that 
$\dim \overline{\phi(C)} \ge 2$ for each $C \in \Irr X$. 
For $f$ in a set of density~$1$,
there is a bijection $\Irr X \to \Irr_{\horiz} X_f$ 
sending $C$ to $(C \intersect X_f)_{\horiz}$.
\end{theorem}

Alternatively, we can obtain a result for all irreducible components,
but with only positive density instead of density~$1$:

\begin{corollary}
Under the hypotheses of Theorem~\ref{T:inverse image bijection}, 
if $V$ is the set of closed points $P \in \PP^n_{\F_q}$
such that $\phi^{-1} P$ is of codimension~$1$ in $X$, 
then the density of the set of $f$ such that 
there is a bijection $\Irr X \to \Irr X_f$ 
sending $C$ to $C \intersect X_f$
is $\prod_{P \in V} \left(1 - q^{-\deg P} \right)$.
\end{corollary}

\begin{proof}
By Lemma~\ref{L:vanishing on X},
we may assume that $f$ does not vanish on any irreducible component of $X$.
Then $\Irr_{\horiz} X_f = \Irr X_f$ 
if and only if $f$ does not vanish at any element of~$V$.
\end{proof}

\subsection{Strategy of proofs}

To prove a statement such as Theorem~\ref{T:geometrically irreducible}
for a variety $X$,
we bound the number of geometrically reducible divisors of $X$ cut out by
hypersurfaces of large degree $d$.
It is easier to count decompositions into effective Cartier divisors
(as opposed to arbitrary subvarieties)
since we can count sections of line bundles.
So it would be convenient if $X$ were smooth.
If $\dim X = 2$, then a resolution of singularities $\Xtilde \to X$
is known to exist, and we can relate the problem for $X$ to a 
counting problem on $\Xtilde$: see Proposition~\ref{P:2-dimensional case}.
If $\dim X > 2$, then resolution of singularities for $X$ might not be known,
and using an alteration seems to destroy the needed bounds,
so instead we use induction on the dimension.
The idea is to apply the inductive hypothesis to $J \intersect X$
for some hypersurface $J$, but this requires $J \intersect X$
itself to be geometrically irreducible,
and finding such a $J$ 
is dangerously close to the original problem we are trying to solve.
Fortunately, finding \emph{one} such $J$ of unspecified degree is enough,
and this turns out to be easier: 
see Lemma~\ref{L:hypersurface section}\eqref{I:part 1}.

\subsection{Applications}

Two applications for our theorems existed 
even before the theorems were proved: 
\begin{itemize}
\item 
Alexander Duncan and Zinovy Reichstein
observed that 
Theorem~\ref{T:inverse image bijection} can be used to extend Theorem~1.4 of 
their article~\cite{Duncan-Reichstein-preprint}
to the case of an arbitrary ground field $k$;
originally they proved it only over an infinite $k$.
Their theorem compares variants of the notion of essential dimension
for finite subgroups of $\GL_n(k)$.
They need to use the Bertini theorems to construct hypersurface sections 
passing through a finite set of points 
\cite{Duncan-Reichstein-preprint}*{Theorem~8.1}.
\item 
Ivan Panin 
observed that Theorem~\ref{T:geometrically irreducible} 
can be used to extend a result concerning 
the Grothendieck--Serre conjecture on principal $G$-bundles.
The conjecture states that 
if $R$ is a regular local ring, $K$ is its fraction field,
and $G$ is a reductive group scheme over $R$,
then $H^1_{\et}(R,G) \to H^1_{\et}(K,G)$ has trivial kernel.
In \cites{Panin-preprint1,Panin-preprint2,Panin-preprint3}, 
Panin proves such a statement for regular semi-local domains
containing a field $k$;
his proof relies on Theorem~\ref{T:geometrically irreducible} 
when $k$ is finite.
\end{itemize}

Here is another application, 
in the same spirit as \cite{Duncan-Reichstein-preprint}*{Theorem~8.1}.
By \defi{variety} we mean a separated finite-type scheme over a field.

\begin{theorem}
\label{T:variety through finite set}
Let $X$ be a geometrically irreducible variety of dimension $m \ge 2$ 
over a field $k$.
Let $F \subset X$ be a finite set of closed points.
Then there exists a geometrically irreducible variety 
$Y \subseteq X$ of dimension $m-1$ containing $F$.
\end{theorem}

\begin{proof}
By Nagata's embedding theorem (see, e.g., \cite{Conrad2007}),
$X$ embeds as a dense open subscheme of a proper $k$-scheme $\overline{X}$.
If we find a suitable $Y$ for $(\overline{X},F)$ ,
then $Y \intersect X$ solves the problem for $(X,F)$
(if necessary, we enlarge $F$ to be nonempty 
to ensure that $Y \intersect X$ is nonempty).
So we may assume that $X$ is proper.

Chow's lemma provides a projective variety $X'$ and a birational
morphism $\pi \colon X' \to X$.
Enlarge $F$, if necessary, 
so that $F$ contains a point in an open subscheme $U \subseteq X$ 
above which $\pi$ is an isomorphism.
Choose a finite set of closed points $F'$ of $X'$ such that $\pi(F')=F$.
If $Y'$ solves the problem for $(X',F')$,
then $\pi(Y')$ solves the problem for $(X,F)$
(we have $\dim \pi(Y') = \dim Y'$ since $Y'$ meets $\pi^{-1} U$).
Thus we may reduce to the case in which $X$ is embedded in a projective space.

When $k$ is infinite, one can use the classical Bertini irreducibility
theorem as in \cite{MumfordAV1970}*{p.~56} to complete the proof.
So assume that $k$ is finite.

We will let $Y \colonequals H_f \intersect X$ for some $f$ of high degree.
For $f$ in a set of positive density,
$H_f$ contains $F$, 
and Theorem~\ref{T:geometrically irreducible} 
shows that for $f$ outside a density $0$ set,
$H_f \intersect X$ is geometrically irreducible (and of dimension $m-1$). 
As a consequence, we can find $f$ such that $H_f$ satisfies both conditions.
\end{proof}

\begin{corollary}
For $X$, $m$, and $F$ as in Theorem~\ref{T:variety through finite set},
and for any integer $y$ with $1 \le y \le m$,
there exists a $y$-dimensional geometrically irreducible variety 
$Y \subseteq X$ containing $F$.
\end{corollary}

\begin{proof}
Use Theorem~\ref{T:variety through finite set} iteratively.
\end{proof}

\subsection{An anti-Bertini theorem}

The following theorem uses a variant of the Bertini smoothness theorem
to provide counterexamples
to the Bertini irreducibility theorem over finite fields
if we insist on \emph{hyperplane} sections
instead of higher-degree hypersurface sections. 
This discussion parallels \cite{Poonen2004-bertini}*{Theorem~3.1}.

\begin{theorem}[Anti-Bertini theorem]
Fix a finite field $\F_q$.
For every sufficiently large positive integer $d$,
there exists a geometrically irreducible degree~$d$ 
surface $X \subseteq \PP^3_{\F_q}$ 
such that $H \intersect X$ is reducible for every $\F_q$-plane 
$H \subseteq \PP^3_{\F_q}$.
\end{theorem}

\begin{proof}
Let $P = \PP^3_{\F_q} - \PP^3(\F_q)$.
Let $Z$ be the union of all the $\F_q$-lines in $\PP^3$.
By \cite{Poonen2008-subvariety}*{Theorem~1.1(i)} applied to $P$ and $Z$,
for any sufficiently large $d$, 
there exists a degree~$d$ surface $X \subseteq \PP^3_{\F_q}$ containing $Z$ 
such that $X \intersect P$ is smooth of dimension~$2$.
If $X_\F$ were reducible, then $X_\F$ would be singular along
the intersection of two of its irreducible components;
such an intersection would be of dimension at least~$1$,
so this contradicts the smoothness of $X \intersect P$.
Thus $X$ is geometrically irreducible.

On the other hand, for any $\F_q$-plane $H \subseteq \PP^3_{\F_q}$,
the $1$-dimensional intersection $H \intersect X$ 
contains all the $\F_q$-lines in $H$,
so $H \intersect X$ is reducible.
\end{proof}

\begin{remark}
One could similarly construct, for any $d_0$, 
examples in which no hypersurface section
of degree less than or equal $d_0$ is irreducible:
just include higher degree curves in $Z$.
Also, one could give higher-dimensional examples,
hypersurfaces $X$ in $\PP^n_{\F_q}$ for $n>3$
containing all $(n-2)$-dimensional $\F_q$-subspaces $L$: 
a straightforward generalization of
\cite{Poonen2008-subvariety}*{Theorem~1.1(i)} 
proves the existence of such $X$ whose singular locus 
has codimension~$2$ in $X$.
\end{remark}

\section{Notation}\label{S:notation}

If $x$ is a point of a scheme $X$, let $\kappa(x)$ be its residue field.
If $X$ is an irreducible variety, let $\kappa(X)$ be its function field.
If $L \supseteq k$ is an extension of fields, and $X$ is a $k$-scheme, 
let $X_L$ be the $L$-scheme $X \times_{\Spec k} \Spec L$.
If $X$ is a reduced subscheme of a projective space,
let $\overline{X}$ be its Zariski closure.
Given a finite-type scheme $X$ over a field,
let $X_{\red}$ be the associated reduced subscheme, 
let $X^{\smooth}$ be the smooth locus of $X$,
and let $X^{\sing}$ be the closed subset $X \setminus X^{\smooth}$.

\section{Lemmas}\label{S:lemmas}

\begin{lemma}
\label{L:vanishing on X}
Let $X$ be a positive-dimensional subscheme of $\PP^n$ over $\F_q$ or $\F$.
The density of 
$\{ f \in S_{\homog} : f \textup{ vanishes on $X$}\}$ 
is $0$.
\end{lemma}

\begin{proof}
We may assume that $X$ is over $\F_q$;
if instead $X$ is over $\F$, replace $X$ by its image
under $\PP^n_{\F} \to \PP^n_{\F_q}$.
If $x$ is a closed point of $X$,
the upper density of the set of $f$ vanishing on $X$
is bounded by the density of the set of $f$ vanishing at $x$,
which is $1/\#\kappa(x)$.
Since $X$ is positive-dimensional, 
we may choose $x$ of arbitrarily large degree.
\end{proof}

\begin{lemma}
\label{L:intersects X}
Let $X$ be a positive-dimensional subscheme of $\PP^n$ over $\F_q$ or $\F$.
The density of $\{ f \in S_{\homog} : H_f \intersect X \ne \emptyset \}$ 
is $1$.
\end{lemma}

\begin{proof}
Again we may assume that $X$ is over $\F_q$.
Given $r \in \R$,
let $X_{<r}$ be the set of closed points of $X$ of degree $<r$.
The density of $\{f \in S_{\homog} : H_f \intersect X_{<r} = \emptyset \}$ 
equals the finite product
\[
	\prod_{P \in X_{<r}} \left(1 - q^{-\deg P} \right) 
	= \zeta_{X_{<r}}(1)^{-1},
\]
which diverges to $0$ as $r \to \infty$,
since $\zeta_X(s)$ converges only for $\re(s)>\dim X$.
\end{proof}

\begin{lemma}
\label{L:open subscheme}
Let $X$ and $\phi$ be as in 
Theorem~\ref{T:inverse image bijection}.
Let $U \subseteq X$ be a dense open subscheme.
For $f$ in a set of density~$1$,
there is a bijection $\Irr_{\horiz} X_f \to \Irr_{\horiz} U_f$ 
sending $D$ to $D \intersect U$.
\end{lemma}

\begin{proof}
Lemma~\ref{L:vanishing on X} shows that 
the set of $f$ vanishing on $\overline{\phi(D)}$
for some $D \in \Irr_{\horiz}(X \setminus U)$ has density~$0$.
After excluding such $f$,
every $D \in \Irr_{\horiz} X_f$ meets $U$
(if not, 
then $D \in \Irr_{\horiz}(X \setminus U)$ and $f(\overline{\phi(D)})=0$).
Then the sets $\Irr_{\horiz} X_f$ and $\Irr_{\horiz} U_f$ 
are in bijection: 
the forward map sends $D$ to $D \intersect U$,
and the backward map sends $D$ to its closure in $X_f$.
\end{proof}

\begin{lemma}
\label{L:restatement of bijection}
Let $X$ be a smooth finite-type $\F$-scheme 
with a morphism $\phi \colon X \to \PP^n_\F$
such that $\dim \overline{\phi(C)} \ge 2$ for every $C \in \Irr X$.
Let $f \in S_{\homog} \setminus \{0\}$.
The following are equivalent:
\begin{enumerate}[\upshape (a)]
\item There is a bijection $\Irr X \to \Irr_{\horiz} X_f$ 
sending $C$ to $(C_f)_{\horiz} = (C \intersect X_f)_{\horiz}$.
\item For every $C \in \Irr X$, the scheme $(C_f)_{\horiz}$ is irreducible.
\end{enumerate}
\end{lemma}

\begin{proof}\hfill

(a)$\Rightarrow$(b): If the map is defined, 
then each $(C_f)_{\horiz}$ is irreducible by definition.

(b)$\Rightarrow$(a): 
The assumption~(b) implies that the map in~(a) is defined.
It is surjective since any irreducible component of $X_f$
is contained in an irreducible component of $X$.
Smoothness implies that the irreducible components of $X$ are disjoint,
so the map is injective.
\end{proof}

\begin{lemma}
\label{L:singular locus has dim 0}
Let $X$ be a subscheme of $\PP^n_\F$ that is smooth of dimension $m$.
For $f$ in a set of density~$1$,
the singular locus $(X_f)^{\sing}$ is finite.
\end{lemma}

\begin{proof}
We follow the proof of \cite{Poonen2004-bertini}*{Lemma~2.6 and Theorem~3.2}.
The formation of $(X_f)^{\sing}$ is local,
so we may replace $\PP^n_\F$ by $\Aff^n_\F$,
and replace each $f$ by the corresponding dehomogenization 
in the $\F_q$-algebra $A \colonequals \F_q[x_1,\ldots,x_n]$.
For $d \ge 1$, let $A_{\le d}$ be the set of $f \in A$ 
of total degree at most $d$.
Let $\TT_X$ be the tangent sheaf of $X$;
identify its sections with derivations.

Let $X_1,\ldots,X_n$ be the distinct $\Gal(\F/\F_q)$-conjugates of $X$ 
in $\Aff^n_\F$.
For each nonnegative integer $k$, let $B_k$ be set of points in $\Aff^n(\F)$
contained in exactly $k$ of the $X_i$.
For $x \in B_k$, let $V_x$ be the span of the tangent spaces $T_x X_i$
in $T_x \Aff^n_\F$ for all $X_i$ containing $x$.
For $r \ge m$, define $C_{k,r} \colonequals \{x \in B_k : \dim V_x = r\}$
(it would be empty for $r<m$).
Each $B_k$ and each $C_{k,r}$ is the set of $\F$-points of
a locally closed subscheme of $\Aff^n_{\F_q}$;
from now on, $B_k$ and $C_{k,r}$ refer to these subschemes. 
There is a rank~$r$ subbundle $\VV$ of $\TT_{\Aff^n_{\F_q}}|_{C_{k,r}}$
whose fiber at $x \in C_{k,r}(\F)$ is $V_x$.
If $x \in (X_f)^{\sing}(\F) \intersect C_{k,r}(\F)$,
and $y$ is its image in $\Aff^n_{\F_q}$,
then $T_y H_f \supseteq \VV \tensor \kappa(y)$
as subspaces of $T_y \Aff^n_{\F_q}$.

Suppose that $y$ is a closed point of the $\F_q$-scheme $C_{k,r}$.
Choose global derivations $D_1,\ldots,D_r \colon A \to A$
whose images in $\TT_{\Aff^n_{\F_q}}|_{C_{k,r}}$ form a basis for $\VV$
on some neighborhood $U$ of $y$ in $C_{k,r}$.
{}From now on, we use only $D_1,\ldots,D_m$.
Choose $t_1,\ldots,t_m \in A$ such that $D_i(t_j)$ is nonzero at $x$
if and only if $i=j$.
After shrinking $U$, we may assume that 
the values $D_i(t_i)$ are invertible on $U$.
Let $U_X \colonequals U_\F \intersect X$.
If $P \in U_X \intersect (X_f)^{\sing}$,
then $f(P)=(D_1 f)(P) = \cdots = (D_m f)(P)=0$.

Let $\tau=\max_i(\deg t_i)$ and $\gamma=\lfloor (d-\tau)/p \rfloor$.
Select $f_0 \in A_{\le d}$, $g_1 \in A_{\le \gamma}$, \dots, 
$g_m \in A_{\le \gamma}$ 
uniformly and independently at random,
and define
\[
	f \colonequals f_0 + g_1^p t_1 + \dots + g_m^p t_m.
\]
Then the distribution of $f$ is uniform over $A_{\le d}$,
and $D_i f = D_i f_0 + g_i^p (D_i t_i)$ for each $i$.
For $0 \le i \le m$, define the $\F$-scheme
\[
	W_i \colonequals U_X \intersect \{D_1f = \cdots = D_if = 0\}.
\]
Thus $U_X \intersect (X_f)^{\sing} \subseteq W_m$.

In the remainder of this proof, 
the big-$O$ and little-$o$ notation indicate the behavior as $d \to \infty$,
and the implied constants may depend on $n$, $X$, $U$, and the $D_i$,
but not on $f_0$ or the $g_i$ (and of course not on $d$).
We claim that for $0 \le i \le m-1$,
conditioned on a choice of $f_0$, $g_1$, \dots, $g_i$
for which $\dim W_i \le m-i$,
the probability that $\dim W_{i+1} \le m-i-1$ is $1-o(1)$ as $d \to \infty$.
First, let $Z_1,\ldots,Z_\ell$ be the 
$(m-i)$-dimensional irreducible components of $(W_i)_{\red}$.
As in the proof of \cite{Poonen2004-bertini}*{Lemma~2.6},
we have $\ell = O(d^i)$ by B\'ezout's theorem,
and the probability that $D_{i+1} f$ vanishes on a given $Z_j$
is at most $q^{-\gamma-1}$,
so the probability that the inequality $\dim W_{i+1} \le m-i-1$ \emph{fails} 
is at most $\ell q^{-\gamma-1} = o(1)$,
as claimed.

By induction on $i$, the previous paragraph proves that 
$\dim W_i \le m-i$ with probability $1-o(1)$ as $d \to \infty$, 
for each $i$.
In particular, $W_m$ is finite with probability $1-o(1)$.
Thus $U_X \intersect (X_f)^{\sing}$ is finite with probability $1-o(1)$.
Finally, $X$ is covered by $U_X$ for finitely many $U$
(contained in different $C_{k,r}$),
so $(X_f)^{\sing}$ is finite with probability $1-o(1)$.
\end{proof}

\begin{lemma}
\label{L:divisibility}
Let $L \supseteq k$ be a Galois extension of fields.
Let $\phi \colon V \to W$ be a morphism of irreducible $k$-varieties.
If $W$ is normal, then $\# \Irr W_L$ divides $\# \Irr V_L$.
\end{lemma}

\begin{proof}
Let $G=\Gal(L/k)$.
Let $V_0 \in \Irr V_L$.
Since $W$ is normal, 
$W_L$ is normal by \cite{Raynaud1970henseliens}*{VII, Proposition~2},
so the irreducible components of $W_L$ are disjoint.
Thus $\phi(V_0) \subseteq W_0$ for a unique $W_0 \in \Irr W_L$.
For the action of $G$ on $\Irr V_L$ and $\Irr W_L$,
the stabilizers satisfy $\Stab V_0 \subseteq \Stab W_0$.
Thus $(G:\Stab W_0)$ divides $(G:\Stab V_0)$.
Since $W$ is irreducible, $G$ acts transitively on $\Irr W_L$,
so $(G:\Stab W_0) = \# \Irr W_L$, 
and likewise $(G:\Stab V_0) = \# \Irr V_L$.
\end{proof}

\section{Surfaces over a finite field}\label{S:surfaces}

\begin{proposition}
\label{P:2-dimensional case}
Let $X$ be a $2$-dimensional closed integral subscheme of $\PP^n_{\F_q}$.
For $f$ in a set of density~$1$, there is a bijection 
$\Irr X_{\F} \to \Irr (X_f)_\F$ sending $C$ to $C \intersect X_f$.
\end{proposition}

Before beginning the proof of Proposition~\ref{P:2-dimensional case},
we prove a lemma.

\begin{lemma}
\label{L:dimension of space of sections}
Let $Y$ be a smooth projective irreducible surface over a field $k$.
Let $B$ be a big and nef line bundle on $Y$
(see, e.g., \cite{Lazarsfeld2004vol1}*{Definitions 2.2.1 and 1.4.1}).
Then for every line bundle $L$ on $Y$,
\[
	h^0(Y,L) \le \frac{(L.B)^2}{2 B.B} + O(L.B) + O(1),
\]
where the implied constants depend on $Y$ and $B$, but not $L$.
\end{lemma}

\begin{proof}
If $C$ is an effective curve on $Y$, then $C.B \ge 0$ since $B$ is nef.
In particular, if $L$ has a nonzero section, then $L.B \ge 0$.
Thus if $L.B<0$, then $h^0(Y,L)=0$.

Now suppose that $L.B \ge 0$.
Since $B$ is big, we may replace $B$ by a power to assume that
$B = \OO(D)$ for some effective divisor $D$.
Taking global sections in
\[
	0 \to L(-D) \to L \to L|_D \to 0
\]
yields
\begin{equation}
\label{E:h^0 inequality}
	h^0(Y,L) \le h^0(Y,L(-D)) + h^0(D,L|_D) = 
	h^0(Y,L(-D)) + L.B + O(1),
\end{equation}
while $L(-D).B = L.B - B.B$.
Combine \eqref{E:h^0 inequality} for
$L$, $L(-D)$, $L(-2D)$, \dots until reaching $L(-mD)$
such that $L(-mD).B < 0$;
then $h^0(Y,L(-mD))=0$ by the first paragraph.
We have $m \le \lfloor L.B/B.B \rfloor + 1$,
and the result follows by summing an arithmetic series.
\end{proof}

\begin{proof}[Proof of Proposition~\ref{P:2-dimensional case}]
Let $\pi \colon \Xtilde \to X$ be a resolution of singularities of $X$.
Let $B$ be a divisor with $\OO(B) \isom \pi^* \OO_X(1)$.
Then $B$ is big and nef, so there exists $b \in \Z_{>0}$
such that $bB$ is linearly equivalent 
to $A+E$ with $A$ ample and $E$ effective.
Consider $f \in S_d$.
By Lemma~\ref{L:vanishing on X}, 
we may discard the $f$ vanishing on any one curve in $X$.
In particular, given a positive constant $d_0$, 
we may assume that $H_f$ does not contain any $1$-dimensional
irreducible component of $\pi(E)$
or any of the finitely many curves $C$ on $X$ with $C.\OO_X(1) < d_0$.
(In fact, we could have chosen $A$ and $E$ so that $\pi(E)$ is finite.)

The linear map
\[
	H^0(\PP^n_{\F_q},\OO(d)) \to H^0(X,\OO_X(d))
\]
is surjective for large $d$, and $h^0(X,\OO_X(d)) \to \infty$.
Thus densities can be computed by counting divisors $X_f$ 
(corresponding to elements of $\PP H^0(X,\OO_X(d))$)
instead of polynomials $f$.

If $X_f$ is reducible, then the Cartier divisor $\pi^* X_f$ 
may be written as $D+D'$
where $D$ and $D'$ are effective divisors
such that $D.B, D'.B \ge d_0$ and $D$ is irreducible 
(and hence horizontal relative to $\pi$ since $D.B>0$).
Let $L$ be the line bundle $\OO(D)$ on $\Xtilde$.
Since $\OO(\pi^* X_f) \isom \OO(d B)$,
we have $\OO(D') \isom \OO(dB) \tensor L^{-1}$.
Let $n \colonequals L.B$.
Then $d_0 \le n \le d B.B - d_0$.

We will bound the number of reducible $X_f$ 
by bounding the number of possible $L$ for each $n$,
and then bounding the number of pairs $(D,D')$ for a fixed $L$.
The assumption on $H_f$ implies that $D$ and $E$ have no common components,
so $L.E = D.E \ge 0$.
Thus $L.A \le L.(A+E) = L.(bB) = nb$.
The numerical classes of effective $L$ are lattice points 
in the closed cone $\overline{\NE}(\Xtilde) \subseteq (N_1 \Xtilde)_{\R}$, 
on which $A$ is positive (except at $0$),
so the number of such $L$ up to numerical equivalence 
satisfying $L.A \le nb$ is $O(n^\rho)$ as $n \to \infty$ for some $\rho$.
The number of $L$ within each numerical equivalence 
is at most $\# \PIC^\tau_{\Xtilde}(\F_q)$,
where $\PIC^{\tau}_{\Xtilde}$ is the finite-type subgroup scheme
of the Picard scheme parametrizing line bundles 
with torsion N\'eron--Severi class.
Thus the number of possible $L$ is $O(n^\rho)$.

Now fix $L$.
Recall that $n=L.B$.
By Lemma~\ref{L:dimension of space of sections} with $Y \colonequals \Xtilde$,
\begin{align*}
	& h^0(\Xtilde,L) + h^0(\Xtilde,\OO(dB) \tensor L^{-1}) \\
	&\le \frac{(L.B)^2}{2 B.B} + O(L.B) + O(1) 
		+ \frac{(dB.B-L.B)^2}{2 B.B} + O(dB.B-L.B) 
		+ O(1) \\
	&\le \frac{B.B}{2} d^2 - \frac{n(d B.B - n)}{B.B} + O(d).
\end{align*}
Summing over all $n \in [d_0,d B.B-d_0]$ and all $L$ shows that
the number of pairs $(D,D')$ is at most
\begin{align*}
	\sum_{n=d_0}^{d B.B -d_0} O(n^\rho) q^{h^0(\Xtilde,L) + h^0(\Xtilde,\OO(dB) \tensor L^{-1})}
	&\le O(d^\rho) \sum_{n=d_0}^{d B.B -d_0} q^{\frac{B.B}{2} d^2 - \frac{n(d B.B - n)}{B.B} + O(d)} \\
	&\le d^\rho q^{\frac{B.B}{2} d^2 + O(d)} \; 2 \sum_{n=d_0}^{d B.B/2} q^{- \frac{n(d B.B - n)}{B.B}} \\
	&\le q^{\frac{B.B}{2} d^2 + O(d)} \sum_{n=d_0}^{d B.B/2} q^{-nd/2} \\
	&\le q^{\frac{B.B}{2} d^2 - \frac{d_0 d}{2} + O(d)}.
\end{align*}
The number of reducible $X_f$ is at most this.
On the other hand, the total number of $X_f$ is
\begin{equation}
\label{E:number of X_f}
	\# \PP H^0(X,\OO_X(d)) = q^{\frac{B.B}{2} d^2 + O(d)} 
\end{equation}
since $\deg X = B.B$.
Dividing yields a proportion that tends to $0$ as $d \to \infty$,
provided that $d_0$ was chosen large enough.

\bigskip

Finally we must bound the number of irreducible $X_f$ such that 
the conclusion of Proposition~\ref{P:2-dimensional case} fails,
i.e., there is not a bijection $\Irr X_\F \to \Irr(X_f)_\F$ 
sending $C$ to $C \intersect X_f$.
Consider such an $f$.
Let $Y \in \Irr \Xtilde_\F$.
Let $\F_r$ be the field of definition of $Y$.
Redefine $Y$ as an element of $\Irr \Xtilde_{\F_r}$.
If we view the $\F_r$-scheme $Y$ as an $\F_q$-scheme
(by composing with $\Spec \F_r \to \Spec \F_q$), 
then the morphisms $Y \to \Xtilde \to X$ are birational $\F_q$-morphisms.
Thus $Y \times_{\F_q} \F$ and $X_\F$ share a common smooth dense open subscheme.
Applying Lemma~\ref{L:open subscheme} twice lets us deduce that, 
excluding $f$ in a set of density~$0$, 
there still is not a bijection 
$\Irr(Y \times_{\F_q} \F) \to \Irr_{\horiz}(Y_f \times_{\F_q} \F)$ 
sending $C$ to $(C_f)_{\horiz} = (C \intersect Y_f)_{\horiz}$.
Then, by Lemma~\ref{L:restatement of bijection} applied to $Y \times_{\F_q} \F$,
there exists $C \in \Irr(Y \times_{\F_q} \F)$ such that $(C_f)_{\horiz}$ 
is reducible.
But $C$ is the base change of $Y$ 
by \emph{some} $\F_q$-homomorphism $\F_r \to \F$,
so the $\F_r$-scheme $(Y_f)_{\horiz}$ is not geometrically irreducible.
On the other hand, by Lemma~\ref{L:intersects X}, 
for $f$ in a set of density~$1$, 
the scheme $X_f$ meets $X^{\smooth}$,
in which case $(Y_f)_{\horiz}$ viewed as $\F_q$-scheme is birational to $X_f$,
so $(Y_f)_{\horiz}$ is irreducible.
Thus we have a map from the set $\XX$ of irreducible $X_f$ 
such that the conclusion of Proposition~\ref{P:2-dimensional case}
fails (excluding the density $0$ sets above)
to the set $\YY$ of irreducible and geometrically reducible schemes
of the form $Z_{\horiz}$ for $Z \in \PP H^0(Y,\OO(d))$,
namely the map sending $X_f$ to $(Y_f)_{\horiz}$.
This map is injective since $X_f$ is determined as a Cartier divisor 
by any of its dense open subschemes, and $X_f$ and $(Y_f)_{\horiz}$ 
share a subscheme that is dense and open in both.
Therefore it suffices to bound $\# \YY$.

For $e \ge 2$, let $\YY_e$ be the set of $Z_{\horiz} \in \YY$
such that there exists an effective Cartier divisor $D$ on $Y_{\F_{r^e}}$
such that
\[
	(Z_{\F_{r^e}})_{\horiz} = \sum_{\sigma \in \Gal(\F_{r^e}/\F_r)} {}^\sigma \! D.
\]
Then $\YY = \Union_{e \ge 2} \YY_e$.

Let $L$ be the line bundle $\OO(D)$ on $Y_{\F_{r^e}}$.
Let $B_Y = B|_Y$.
Then $B.B = [\F_r:\F_q] B_Y.B_Y$ 
since the local self-intersection numbers of $B$
have the same sum on each conjugate of $Y$.
Let $n \colonequals L.B_Y$.
Then $ne = Z_{\horiz}.B_Y = Z.B_Y = d B_Y.B_Y$,
so $n = (d/e) B_Y.B_Y$.
The number of numerical classes of effective $L$ with $L.B_Y = n$
is $O(n^\rho)$ as before, uniformly in $e$ (because they are so bounded
even over $\F$).
We have $\# \PIC^\tau_Y(\F_{r^e}) = (r^e)^{O(1)}$.
Thus the number of possible $L$ is $O(n^\rho) (r^e)^{O(1)}$.
Applying Lemma~\ref{L:dimension of space of sections} to $Y_\F$ shows that 
\begin{align*}
	h^0(Y_{\F_{r^e}},L) 
	&= h^0(Y_{\F},L) \\
	&\le \frac{(L.B_Y)^2}{2 B_Y.B_Y} + O(L.B_Y) + O(1) \\
	&\le \frac{(d/e)^2 (B_Y.B_Y)^2}{2 B_Y.B_Y} + O(d/e) \\
	&\le \frac{(d/e)^2 B_Y.B_Y}{2} + O(d/e) \\
	&\le \frac{(d/e)^2 B.B}{2 [\F_r:\F_q]} + O(d/e),
\end{align*}
so
\[
	\# H^0(Y_{\F_{r^e}},L)
	= (r^e)^{h^0(Y_{\F_{r^e}},L)} 
	\le q^{d^2 B.B/2e + O(d)}
\]
since $r=O(1)$, and
\[
	\# \YY_e \le O(n^\rho) (r^e)^{O(1)} q^{d^2 B.B/2e + O(d)}
	\le q^{d^2 B.B/2e + O(d)}
\]
since $n$ and $e$ are $O(d)$, so
\[
	\# \YY = \sum_{e=2}^{O(d)} \# \YY_e 
	\le O(d) q^{d^2 B.B/4 + O(d)}.
\]
This divided by the quantity~\eqref{E:number of X_f} tends to $0$ 
as $d \to \infty$.
\end{proof}

\section{Reductions}\label{S:reductions}

\begin{lemma}
\label{L:finite etale}
Let $X$ and $Y$ be irreducible finite-type $\F$-schemes,
with morphisms $X \stackrel{\pi}\to Y \stackrel{\psi}\to \PP^n_\F$
such that $\pi$ is finite \'etale, 
$\psi \colon Y \to \overline{\psi(Y)}$ is smooth of relative dimension $s$,
and $\dim \overline{\psi(Y)} \ge 2$.
For $f$ in a set of density~$1$,
the implication
\[
	\textup{$Y_f$ irreducible $\implies$ $X_f$ irreducible}
\]
holds.
\end{lemma}

(The reverse implication holds for \emph{all} $f$.  Later
we will prove that both sides hold for $f$ in a set of density~$1$.)

\begin{proof}[Proof of Lemma~\ref{L:finite etale}]
Since irreducibility is a purely topological property,
we may replace $Y$ by $Y_{\red}$
and $\pi \colon X \to Y$ by its pullback to $Y_{\red}$;
then $X$ is reduced too.
Irreducibility of $X_f$ only becomes harder to achieve 
if $X$ is replaced by a higher finite \'etale cover of $Y$.
In particular, we may replace $X$ by a cover corresponding
to a Galois closure of $\kappa(X)/\kappa(Y)$.
So assume from now on that $X \to Y$ is Galois \'etale,
say with Galois group $G$.

Choose a finite extension $\F_r$ of $\F_q$
with a morphism $\psi' \colon Y' \to \PP^n_{\F_r}$ 
and a Galois \'etale cover $\pi' \colon X' \to Y'$ 
whose base extensions to $\F$ 
yield $\psi$ and $\pi$.
Let $Z' \colonequals \overline{\psi'(Y')}$.
Let $m \colonequals \dim Z' = \dim \overline{\psi(Y)} \ge 2$.
Then $\dim Y' = \dim Y = s+m$.
The morphism $\psi' \colon Y' \to Z'$ is smooth,
so it maps $(Y')^{\smooth}$ into $(Z')^{\smooth}$.

Given a closed point $y \in Y'$,
let $\Frob_y$ be the associated Frobenius conjugacy class in $G$.
We will prove that the following claims hold for $f$ in a set of density~$1$:
\begin{enumerate}[\phantom{m}{Claim}~1.]
\item The $\Frob_y$ for $y \in (Y_f')^{\smooth}$ 
cover all conjugacy classes of $G$.
\item The $\F_r$-scheme $(X_f')^{\smooth}$ 
contains two closed points whose degrees over $\F_r$ are coprime.
\end{enumerate}
Let $C$ be a conjugacy class in $G$.
Let $c:=\#C/\#G$.
In the arguments below, 
for fixed $X'$, $Y'$, $\psi'$, $\pi'$, $G$, and $C$, 
the expression $o(1)$ 
denotes a function of $e$ that tends to $0$ as $e \to \infty$.
By a function field analogue of the 
Chebotarev density theorem~\cite{Lang1956-Lseries}*{last display on p.~393}
(which, in this setting, follows from applying the Lang--Weil estimates
to all twists of the cover $Y' \to X'$),
the number of closed points $y \in (Y')^{\smooth}$ with residue field $\F_{r^e}$
satisfying $\Frob_y = C$
is $\left(c + o(1) \right)r^{(s+m)e}/e$.
Since each nonempty fiber of $\psi'$ has dimension~$s$,
there exists $c'>0$ such that 
the images of these points in $\PP^n_{\F_q}$ 
are at least $\left(c' + o(1) \right)r^{me}/e$ closed points
$z \in (Z')^{\smooth}$ with residue field of size at most $r^e$.
For any such $z$, say with residue field of size $r^\epsilon$,
the density of $\{f : z \notin H_f\}$
is $1-r^{-\epsilon}$,
and the density of 
$\{f : \textup{$z \in H_f$ and $H_f$ is not transverse to $Z'$ at $z$}\}$
is $r^{-\epsilon} r^{- \epsilon m}$,
so the union of these two disjoint sets has density 
$1-r^{-\epsilon} +r^{-\epsilon(1+m)} \le 1 -r^{-\epsilon}/2 \le 1-r^{-e}/2$.
These conditions at the finitely many $z$ are independent,
so the density of the set $\calQ_{C,e}$ of $f$ such that they hold at all $z$
is at most $\left(1 - r^{-e}/2 \right)^{\left(c' + o(1) \right)r^{me}/e}$,
which tends to $0$ as $e \to \infty$ since $m \ge 2$.
If the condition at some $z$ fails,
then $z \in (Z'_f)^{\smooth}$,
and any $y \in (Y')^{\smooth}$ with residue field $\F_{r^e}$
with $\psi'(y)=z$ lies in $(Y'_f)^{\smooth}$,
since $\psi' \colon Y' \to Z'$ is smooth.
Thus the complement $\calP_{C,e}$ of $\calQ_{C,e}$ equals
the set of $f$ for which there exists $y \in (Y'_f)^{\smooth}$ such that
$\kappa(y) = \F_{r^e}$ and $\Frob_y=C$.
The lower density of $\calP_{C,e}$ tends to $1$ as $e \to \infty$.

\medskip

Proof of Claim~1: There are only finitely many $C$,
so the previous sentence shows that
the lower density of $\Intersection_C \calP_{C,e}$
tends to $1$ as $e \to \infty$.

\medskip

Proof of Claim~2: 
If $f \in \calP_{1,e}$, then there exists $y \in (Y'_f)^{\smooth}$
with $\kappa(y)=\F_{r^e}$ and $\Frob_y=1$,
and any preimage $x \in X'_f$ is a point of $(X'_f)^{\smooth}$
satisfying $\kappa(x)=\F_{r^e}$, since $X_f' \to Y_f'$ is finite \'etale
and $\Frob_y=1$.
The lower density of $\calP_{1,e} \intersect \calP_{1,e'}$ tends to $1$
as $(e,e')$ runs through pairs of coprime integers with 
$\min(e,e') \to \infty$.

\medskip

To complete the proof of the lemma,
we show that if $Y_f$ is irreducible
and Claims~$1$ and~$2$ hold,
then $X_f$ is irreducible.
Assume that $Y_f$ is irreducible, so $Y_f'$ is geometrically irreducible.
The only subgroup of $G$ that meets all conjugacy classes
is $G$ itself, so Claim~1 implies that
$(X'_f)^{\smooth} \to (Y'_f)^{\smooth}$
is a finite Galois \emph{irreducible} cover 
(with Galois group $G$).

If $x'$ is a closed point of $(X'_f)^{\smooth}$ of degree $e$ over $\F_r$,
then applying Lemma~\ref{L:divisibility} to $\F \supseteq \F_r$
and $\{x'\} \hookrightarrow (X'_f)^{\smooth}$
shows that $\# \Irr X_f^{\smooth}$ divides $e$.
Applying this to both points in Claim~2 shows that 
$\# \Irr X_f^{\smooth} = 1$, so $X_f^{\smooth}$ is irreducible.
On the other hand, $Y_f$ is irreducible and $Y_f^{\smooth}$ is nonempty,
so $Y_f^{\smooth}$ is dense in $Y_f$;
since $X_f \to Y_f$ is finite \'etale, $X_f^{\smooth}$ is dense in $X_f$ too.
Combining the previous two sentences shows that $X_f$ is irreducible.
\end{proof}

We now strengthen Lemma~\ref{L:finite etale} to replace finite \'etale
by dominant (but we change the hypothesis on $\psi$ as well
since we will need only the case in which $\psi$ is an immersion).

\begin{lemma}
\label{L:dominant}
Let $X$ and $Y$ be irreducible finite-type $\F$-schemes,
with morphisms $X \stackrel{\pi}\to Y \stackrel{\psi}\to \PP^n_\F$
such that $\pi$ is dominant, $\psi$ is an immersion, 
and $\dim Y \ge 2$.
For $f$ in a set of density~$1$,
the implication
\[
	\textup{$(Y_f)_{\horiz}$ irreducible 
	$\implies$ $(X_f)_{\horiz}$ irreducible}
\]
holds.
\end{lemma}

\begin{proof}
By Lemma~\ref{L:open subscheme},
we may replace $X$ and $Y$ by dense open subschemes;
thus we may assume that $\pi$ factors as
\[
\xymatrix{
	X \ar[rrr]^{\textup{surjective radicial}} 
	&&& V \ar[rr]^{\textup{finite \'etale}} 
	&& W \ar[rrr]^{\textup{dense open immersion}}
	&&& \Aff^r_Y \ar[rr] 
	&& Y
}
\]
for some $r \in \Z_{\ge 0}$,
because $\kappa(X)$
is a finite inseparable extension of a finite separable extension
of a purely transcendental extension of $\kappa(Y)$, 
and we may spread this out.
Then $(Y_f)_{\horiz}=Y_f$ and similarly for $\Aff^r_Y$, $W$, $V$, and $X$.
Irreducibility of $Y_f$ 
is equivalent to irreducibility of $(\Aff^r_Y)_f$;
for $f$ in a set of density~$1$, 
this is equivalent to irreducibility of $W_f$ 
(Lemma~\ref{L:open subscheme}),
which implies irreducibility of $V_f$ 
(Lemma~\ref{L:finite etale} applied to $V \to W \to \PP^n_\F$),
which is equivalent to irreducibility of $X_f$ (homeomorphism).
\end{proof}

Part~\eqref{I:part 1} of the following lemma and its proof 
are closely related to results of 
Lior Bary-Soroker~\cite{Bary-Soroker-letter};
see also \cite{Neumann1998}*{Lemma~2.1}, attributed to Wulf-Dieter Geyer.

\begin{lemma}
\label{L:hypersurface section}
Let $X$ be a smooth irreducible subscheme of $\PP^n_\F$ of dimension $m \ge 3$.
\begin{enumerate}[\upshape (a)]
\item\label{I:part 1}
There exists a hypersurface $J \subseteq \PP^n_{\F_q}$ 
such that $\dim J \intersect (\Xbar \setminus X) \le m-2$

and $J \intersect X$ is irreducible of dimension $m-1$.
\item\label{I:part 2}
For any such $J$,
there exists a density~$1$ set of $f$
for which the implication
\[
	\textup{$(J \intersect X)_f$ irreducible $\implies$ $X_f$ irreducible}
\]
holds.
\end{enumerate}
\end{lemma}

\begin{proof}
\hfill
\begin{enumerate}[\upshape (a)]
\item
Inductively choose $h_0,\ldots,h_m \in S_{\homog}$
so that the common zero locus of $h_0,\ldots,h_r$ on $\Xbar$
is of the expected dimension $m-r-1$ for $r \in \{0,1,\ldots,m-1\}$
and empty for $r=m$.
Replace every $h_i$ by a power to assume that they have the same degree.
Then $(h_0:\cdots:h_m) \colon \PP^n_\F \dashrightarrow \PP^m_\F$
restricts to a morphism $\pi \colon \Xbar \to \PP^m_\F$
whose fiber above $(0:\cdots:0:1)$ is $0$-dimensional,
so $\pi$ is generically finite.
Since $\dim \Xbar = m$, the morphism $\pi$ is dominant.

Let $Z$ be the set of points of $\PP^m_{\F_q}$
above which the fiber has codimension~$1$ in $\Xbar$.
Since $m>1$ and $\pi$ is dominant, $Z$ is finite.
Let $B_1,\ldots,B_s$ be the images in $\PP^m_{\F_q}$
of the irreducible components of $\Xbar \setminus X$.
Let $Z'$ be the union of $Z$ with all the $0$-dimensional $B_i$.
The density of homogeneous polynomials $g$ on $\PP^m_{\F_q}$
such that $H_g \intersect Z' = \emptyset$ 
is $\prod_{z \in Z'} (1- 1/\#\kappa(z)) > 0$.
The density of such $g$ 
such that also $H_g$ is geometrically integral
and does not contain any positive-dimensional $B_i$ is the same
by \cite{Poonen2004-bertini}*{Proposition~2.7} 
and Lemma~\ref{L:vanishing on X}.
For such $g$, the subscheme $X_g \colonequals \pi^{-1} H_g$
is horizontal and contains no irreducible component of $\Xbar \setminus X$.
Lemma~\ref{L:dominant} applied to 
$X \stackrel{\pi}\to \PP^m_{\F} \stackrel{\Id}\to \PP^m_{\F}$
shows that, after excluding a further density~$0$ set,
$X_g$ is irreducible of dimension $m-1$.
Let $J$ be the hypersurface in $\PP^n_{\F_q}$ 
defined by $g(h_0,\ldots,h_m)$.
Then $J$ contains no irreducible component of $\Xbar \setminus X$,
so $\dim J \intersect (\Xbar \setminus X) 
\le \dim(\Xbar \setminus X) - 1 \le m-2$.
Also, $J \intersect X = X_g$, which is irreducible of dimension $m-1$.
\item
Consider $f$ such that 
$f$ does not vanish on any positive-dimensional irreducible component
of $J \intersect (\Xbar \setminus X)$
and $(X_f)^{\sing}$ is finite.
By Lemmas \ref{L:vanishing on X} and~\ref{L:singular locus has dim 0},
this set has density~$1$.

Suppose that $(J \intersect X)_f$ is irreducible
but $X_f$ is reducible, say $X_f = V_1 \union V_2$,
where $V_i$ are closed subsets of $X_f$,
neither containing the other.
Then $\dim V_i \ge m-1$ for each $i$.
Since $J$ is a hypersurface, 
$J \intersect \Vbar_i$ is nonempty and of dimension at least $m-2$.
On the other hand,
$J \intersect (\Vbar_i \setminus V_i) 
\subseteq J \intersect (\Xbar \setminus X)_f$,
which is of dimension at most $m-3$.
Thus $J \intersect V_i$ is nonempty of dimension at least $m-2 \ge 1$.
Since $(J \intersect X)_f = (J \intersect V_1) \union (J \intersect V_2)$,
one of the $J \intersect V_i$ must contain the other,
say $(J \intersect V_2) \subseteq (J \intersect V_1)$.
Then $J \intersect V_2 \subseteq V_1 \intersect V_2 \subseteq (X_f)^{\sing}$.
This is a contradiction since $\dim J \intersect V_2 \ge 1$ 
and $(X_f)^{\sing}$ is finite.\qedhere
\end{enumerate}
\end{proof}

\begin{proposition}
\label{P:1.2 over F_q}
Let $X$ be an irreducible subscheme of $\PP^n_{\F_q}$
of dimension at least $2$.
For $f$ in a set of density~$1$, there is a bijection 
$\Irr X_{\F} \to \Irr (X_f)_\F$ sending $C$ to $C_f$.
\end{proposition}

\begin{proof}
We use induction on $\dim X$.
Replace $X$ by $X_\red$; then $X$ is integral.
The case $\dim X=2$ follows from Proposition~\ref{P:2-dimensional case}
for $\Xbar$, 
by using Lemma~\ref{L:open subscheme} to pass from $X$ to $\Xbar$.

Now suppose $\dim X>2$.
Because of Lemma~\ref{L:open subscheme},
we may shrink $X$ to assume that $X$ is smooth.
Choose one $C \in \Irr X_\F$.
Choose $J$ as in Lemma~\ref{L:hypersurface section} applied to $C$,
so $J \intersect C$ is irreducible.
Then so is its image $J \intersect X$ under $C \to X$.
Also, $\dim(J \intersect X)=\dim X-1 \ge 2$.

We have $J \intersect C \in \Irr((J \intersect X)_\F)$.
The inductive hypothesis applied to $J \intersect X$
shows that for $f$ in a set of density~$1$,
the scheme $(J \intersect C)_f$ is irreducible.
By Lemma~\ref{L:hypersurface section}\eqref{I:part 2}, 
for $f$ in a smaller density~$1$ set, 
this implies that $C_f$ is irreducible.
Each $C' \in \Irr X_\F$ is conjugate to $C$,
so $C'_f$ is irreducible for the same $f$.
For these $f$, Lemma~\ref{L:restatement of bijection} 
implies the conclusion.
\end{proof}

\begin{proposition}
\label{P:1.6 over F_q}
Let $X$ be a finite-type $\F_q$-scheme.
Let $\phi \colon X \to \PP^n_{\F_q}$ be a morphism such that 
$\dim \overline{\phi(C)} \ge 2$ for each $C \in \Irr X$.
For $f$ in a set of density~$1$,
there is a bijection $\Irr X_{\F} \to \Irr_{\horiz} (X_f)_{\F}$
sending $C$ to $(C_f)_{\horiz}$.
\end{proposition}

\begin{proof}
Replace $X$ by $X_\red$ to assume that $X$ is reduced. 
Because of Lemma~\ref{L:open subscheme},
we may shrink $X$ to assume that $X$ is smooth;
then its irreducible components are disjoint.
If the conclusion of Proposition~\ref{P:1.6 over F_q} holds 
for each $C \in \Irr X$, then it holds for $X$.
So assume that $X$ is irreducible.
Let $X' \colonequals \overline{\phi(X)}$.
Consider $C \in \Irr X_\F$.
Let $C' \colonequals \overline{\phi(C)} \in \Irr X'_\F$.
For $f$ in a set of density~$1$,
Proposition~\ref{P:1.2 over F_q} for $X'$ implies irreducibility of $C'_f$,
which, by Lemma~\ref{L:dominant} for $C \to C' \injects \PP^n_\F$,
implies irreducibility of $(C_f)_{\horiz}$.
By Lemma~\ref{L:restatement of bijection},
there is a bijection $\Irr X_\F \to \Irr_{\horiz}(X_f)_{\F}$
sending $C$ to $(C_f)_{\horiz}$.
\end{proof}

\begin{proof}[Proof of Theorem~\ref{T:inverse image bijection}]
We are given a finite-type $\F$-scheme $X$
and a morphism $\phi \colon X \to \PP^n_\F$
such that $\dim \overline{\phi(C)} \ge 2$ for each $C \in \Irr X$.
Let $\F_r$ be a finite extension of $\F_q$
such that $X$, $\phi$, and all irreducible components of $X$ 
are defined over $\F_r$.
{}From now on, consider $X$ and $\phi$ as objects over $\F_r$.
We need to prove that there is a bijection 
$\Irr X_{\F} \to \Irr_{\horiz} (X_f)_{\F}$ 
sending $C$ to $(C \intersect X_f)_{\horiz}$.

The composition $X \to \Spec \F_r \to \Spec \F_q$ lets us reinterpret $X$
as a finite-type $\F_q$-scheme $\calX$ with a morphism $\psi$ to $\PP^n_{\F_q}$ 
fitting in a commutative diagram
\[
\xymatrix{
X \ar@{=}[r] \ar[d]_{\phi} & \calX \ar[d]^{\psi} \\
\PP^n_{\F_r} \ar[r] \ar[d] & \PP^n_{\F_q} \ar[d] \\
\Spec \F_r \ar[r] & \Spec \F_q. \\
}
\]
Since $X$ and $\calX$ are equal as schemes (forgetting the base field),
each irreducible component $C$ of $X$ is an irreducible component $\calC$
of $\calX$.
The morphism $\PP^n_{\F_r} \to \PP^n_{\F_q}$ is finite,
so $\dim \overline{\psi(\calC)} = \dim \overline{\phi(C)} \ge 2$.
Proposition~\ref{P:1.6 over F_q} applied to $\psi$
yields a bijection $\Irr \calX_\F \to \Irr_{\horiz} (\calX_f)_\F$.
We now rephrase this in terms of $X$.
The identification of $X$ with $\calX$
equates $X_f \colonequals \phi^{-1} (H_f)_{\F_r}$ 
with $\calX_f \colonequals \psi^{-1} H_f$.
Then we have a diagram of $\F$-schemes
\[
\xymatrix{
(\calX_f)_\F \ar@{=}[r] \ar[d] & \coprod_\sigma {}^\sigma \! X_f \ar[d] \\
\calX_\F \ar@{=}[r] & \coprod_\sigma {}^\sigma \! X \\
}
\]
where $\sigma$ ranges over $\F_q$-homomorphisms $\F_r \to \F$,
and ${}^\sigma \! X$ denotes the corresponding base extension
(and ${}^\sigma \! X_f$ is similar),
and the vertical map on the right is induced by the inclusion $X_f \injects X$.
Thus for each $\sigma$, 
there is a bijection $\Irr {}^\sigma \! X \to \Irr_{\horiz} {}^\sigma \! X_f$
sending each $C$ to $(C \intersect {}^\sigma \! X_f)_{\horiz}$.
Taking $\sigma$ to be the inclusion $\F_r \injects \F$
yields the conclusion of Theorem~\ref{T:inverse image bijection}.
\end{proof}

\section*{Acknowledgements} 

We thank Ivan Panin and Zinovy Reichstein for independently asking us 
in June 2013 whether Theorem~\ref{T:geometrically irreducible} is true.
We thank also Lior Bary-Soroker for sending us 
his letter~\cite{Bary-Soroker-letter},
in which he developed a new explicit Hilbert irreducibility theorem 
to prove that 
for a geometrically irreducible quasi-projective variety $X$ over $\F_q$
of dimension $r+1 \ge 2$, 
there is \emph{some} dominant rational map 
$\phi \colon X \dashrightarrow \PP^{r+1}_{\F_q}$
such that for a density~$1$ set of hypersurfaces of the form
$y=h(x_1,\ldots,x_r)$ in $\Aff^{r+1}_{\F_q}$,
the closure of its preimage in $X$ is geometrically irreducible.
We thank K\k{e}stutis \v{C}esnavi\v{c}ius for asking a question
that led to Theorem~\ref{T:variety through finite set},
and we thank Jiayu Zhao for pointing out an error 
in an earlier version of Lemma~\ref{L:finite etale},
and for pointing out the relevance of normality in a
statement like Lemma~\ref{L:divisibility}.
Finally, we thank the referees for many helpful suggestions.

\end{document}